\documentclass[10pt]{amsart}
\usepackage{hyperref}
\usepackage{verbatim,url}
\usepackage{amsmath}
\usepackage{enumerate,supertabular}
\usepackage{amssymb}
\usepackage{color}
\usepackage[active]{srcltx}

\numberwithin{equation}{section}

\usepackage[all]{xy}
\usepackage[active]{srcltx}

\theoremstyle{definition}

\newcommand{\Q}{\mathbb{Q}}

\setbox0=\hbox{$+$}
\newdimen\plusheight
\plusheight=\ht 0
\newcommand\+{\;\lower\plusheight\hbox{$+$}\;}

\setbox0=\hbox{$-$}
\newdimen\minusheight
\minusheight=\ht0
\renewcommand\-{\;\lower\minusheight\hbox{$-$}\;}

\setbox0=\hbox{$\cdots$}
\newdimen\cdotsheight
\cdotsheight=\plusheight
\newcommand\cds{\lower\cdotsheight\hbox{$\cdots$}}

\newcommand{\Z}{\mathbb{Z}}

\newtheorem{theorem}{Theorem}

\newtheorem{proposition}[theorem]{Proposition}

\newtheorem{remark}[theorem]{Remark}

\begin{document}

\title{Constructing Class invariants}
\author{  Aristides Kontogeorgis}


\address{
University of Athens, Panepistimioupolis 15784, Athens, Greece} \email{kontogar@math.uoa.gr}

\thanks{This work was supported by the Project 
``{\em Thalis, Algebraic modeling of topological and Computational structures}''.
The Project ``THALIS'' is implemented under the Operational Project
 ``Education and Life Long Learning is co-funded by the European Union
 (European Social Fund) and National Resources (ESPA)}

\begin{abstract}
Shimura reciprocity law allows us to verify that a modular function is a 
class invariant. Here we present a new method based on Shimura 
reciprocity that allows us not only to verify but 
to find new class invariants 
from a modular function 
of level $N$. 
\end{abstract}

\maketitle

\section{Introduction}

It is  well known that the ring class field of imaginary quadratic orders can be generated by evaluating the 
$j$-invariant at certain algebraic integers. There are many modular functions that can be used 
for the generation of the ring class field. In a series of articles \cite{GeeBordeaux},\cite{GeeStevenhagen},
\cite{GeePhD},\cite{steven2}
A. Gee and P. Stevenhagen developed a method based on Shimura reciprocity law, in order to 
check whether a modular function gives rise to a class invariant. 
A necessary condition for this is  the  invariance of the modular function under the action of 
the group $G_N=(\mathcal{O}/N\mathcal{O})^*/\mathcal{O}^*$, where $\mathcal{O}$ is an order of a 
quadratic imaginary field.

So far it seems that all known class invariants were found out of luck or by extremely  ingenious
people like 
Ramanujan. Aim of this article is to provide a systematic method for finding class invariants. 
We will use a combination of techniques from classical invariant theory \cite{invar} and 
 Galois descent \cite{Conrad}.
 
The structure of this article is as follows:
In section \ref{2} we give a very quick description of the technique  based on Shimura 
reciprocity law for checking whether a modular function is a class invariant. The interested 
reader should consider the more detailed explanations found in 
\cite{GeeBordeaux},\cite{GeeStevenhagen},\cite{GeePhD},\cite{steven2}. In section \ref{3} we explain our main observation. 
The action 
of $G_N$ is given in terms of matrices but the function $\rho$ sending elements of the group 
$G_N$ to matrices is not a linear representation but a cocycle. Then we break the computation 
into two parts. The first part considers a subgroup $H$ of $G_N$ such that $\rho$ when 
restricted to $H$ is a linear representation. Classical invariant theory provides us with a set 
of $H$-invariant elements. The second part makes the observation that the quotient
$G_N/H$ is isomorphic to the Galois group $\mathrm{Gal}(\Q(\zeta_N)/\Q)$, 
where $\zeta_N$ is a primitive $N$-th root of unity. Then 
Hilbert 90 ensures us that we can find a set of $G_N$ invariants. 
In section \ref{4} we use our technique  in the case of   generalized Weber functions.
We selected these modular functions since a lot of work has been done on them 
and also the action of $\mathrm{SL}(2,\Z)$ on them is well understood. 
For given $N$ prime number and discriminant $D$  we are able to construct a whole 
$\Q$-vector space consisted
of class invariants. A lot of examples are given and the magma code \cite{magma}
used to compute them is freely available upon request.  

\section{Shimura reciprocity law} \label{2}
Let $\Gamma(N)$ be the kernel of the map 
$\mathrm{SL}(2,\Z) \mapsto \mathrm{SL}\left(2,{\Z}/{N \Z}\right)$. 
The group $\mathrm{SL}(2,\Z)$ acts on the upper half plane $\mathbb{H}$  in terms
of linear fractional transformations and is known to be generated by the 
elements $S:z \mapsto -\frac{1}{z}$ and $T:z\mapsto z+1$. 

It is known that the quotient Riemann surface $\Gamma(N) \backslash \mathbb{H}^*$ can be defined over the field $\mathbb{Q}(\zeta_N)$, where $\zeta_N$ is a 
primitive $N$-th root of unity. We consider the function field 
$F_N$ of the corresponding curve defined over $\mathbb{Q}(\zeta_N)$. The function field 
$F_N$ is acted on 
by 
\[
\Gamma(N)/\{\pm 1\} \cong \mathrm{Gal}(F_N/F_1(\zeta_N)),  
\]
and there is also an arithmetic action of 
\[
 \mathrm{Gal}(F_1(\zeta_N)/F_1) \cong \mathrm{Gal}(\mathbb{Q}(\zeta_N)/\mathbb{Q})
\cong  \left( \frac{\Z}{N\Z} \right)^*.
\]
For an element $d\in \left( \frac{\Z}{N\Z} \right)^*$ we consider the 
automorphism $\sigma_d:\zeta_N \mapsto \zeta_N^d$. 
Since the Fourier coefficients of a function 
$h\in F_N$  are known to be in  $\mathbb{Q}(\zeta_N)$
we consider the action of 
$\sigma_d$ 
 on $F_N$ by applying $\sigma_d$ on the Fourier coeficients of $h$.
We have an action 
of the group 
$\mathrm{GL}\left( 2,\frac{\Z}{N\Z} \right)$ on $F_N$ that fits in the following 
short exact sequence.
\[
 1 \rightarrow \mathrm{SL}\left(2,\frac{\Z}{N \Z}\right) \rightarrow
\mathrm{GL}\left( 2,\frac{\Z}{N\Z} \right) 
\stackrel{\det}{\longrightarrow} \left( \frac{\Z}{N\Z} \right)^*
\rightarrow 1. 
\]
A. Gee and P. Stevehagen  \cite{GeeBordeaux},\cite{GeeStevenhagen},
\cite{GeePhD},\cite{steven2} proved the following theorem 
which was based on the work of Shimura \cite{Shimura}:
\begin{theorem} \label{units2mat}
 Let $\mathcal{O}=\Z[\theta]$ be the ring of integers of an imaginary quadratic number field 
$K$ of discriminant $d< -4$. Suppose that a  modular function $h \in F_N$ does not have a pole 
at $\theta$ and $\mathbb{Q}(j) \subset \mathbb{Q}(h)$. 
Suppose that the minimum polynomial of $\theta$ over $\mathbb{Q}$ is 
$x^2+Bx+C$. Then there is a subgroup 
$W_{N,\theta}\subset \mathrm{GL}\left( 2,\frac{\Z}{N\Z} \right)$ 
with elements of the form:
\[
 W_{N,\theta}=\left\{
\begin{pmatrix}
 t-Bs & -Cs \\
s & t
\end{pmatrix} \in
\mathrm{GL}\left( 2,\frac{\Z}{N\Z} \right): t\theta +s \in 
\left(\mathcal{O}/N\mathcal{O}\right)^* 
\right\}.
\]
The function value $h(\theta)$ is a
class invariant if and only if the group  
$W_{N,\theta}$ acts trivially on $h$.  
\end{theorem}
\begin{proof}
\cite[cor. 4]{GeeBordeaux}.
\end{proof}
The above  theorem can be applied in order to show that a modular function 
gives rise to a class invariant and was used with success in order to 
prove that several functions were indeed class invariants. 
Also A. Gee and P. Stevenhagen provided us with an 
explicit way of describing the Galois action of $\mathrm{Cl}(\mathcal{O})$
on the class invariant so that we can construct the minimal polynomial.

We will now describe an algorithm that will allow us to find class invariants. 
As a result we will obtain a whole $\Q$-vector space of  class invariants.

Let $f$ be a modular function of level $N$. Assume that we can find a finite 
dimensional vector space $V$ consisted of modular functions of level $N$ so that 
$\mathrm{GL}(2,\Z/N\Z)$ acts on $V$. We can always find such a 
vector space. We simple have to consider the orbit of $f$ under the 
action of the finite group $\mathrm{GL}(2,\Z/N\Z)$. Notice that 
every element $a\in \mathrm{GL}(2,\Z/N\Z)$ can be written 
as $b \cdot \begin{pmatrix}
             1 & 0 \\ 0 &d
            \end{pmatrix}$, $d \in \mathbb{Z}/N\Z^*$ and 
$b \in \mathrm{SL}(2,\Z/N\Z)$. 
The group $\mathrm{SL}(2,\Z/N\Z)$ is generated by the elements 
$S=\begin{pmatrix}
    0 & 1 \\ -1 & 0
   \end{pmatrix}$ 
and $T=\begin{pmatrix} 1 & 1 \\ 0 & 1 \end{pmatrix}$. 
The action of $S$ on functions  $g \in V$ is 
defined to be $g \circ S=g(-1/z) \in V$ and the 
action of $T$ is defined $ g \circ T= g(z+1) \in V$. 

Here a technical difficulty arises: 
how can one compute efficiently the decomposition of an element in $\mathrm{SL}(2,\Z/N\Z)$ 
as a product of the generators $S$,$T$?
Observe that by Chinese remainder theorem we can write 
\[
 \mathrm{GL}(2,\Z/N\Z)= \prod_{p \mid N } \mathrm{GL}(2,\Z/ p^{v_p(N)}\Z), 
\]
where $v_p(N)$ denotes the power of $p$ that appears in the decomposition in 
prime factors. Working with the general linear group over a field has advantages 
and one can use lemma 6 in \cite{GeeBordeaux} in order to express an 
element of determinant one in $\mathrm{SL}(2,\Z/ p^{v_p(N)}\Z)$ as 
word in elements $S_p,T_p$ where 
$S_p$ and $T_p$ are $2\times 2$ matrices which  reduce to $S$ and $T$ modulo $p^{v_p(N)}$
and to the identity modulo $q^{v_q(N)}$ for prime divisors $q$ of $N$, $p\neq q$.

This way the problem is reduced to the problem of finding the matrices $S_p,T_p$
(this is easy using the Chinese remainder Theorem), and expressing them 
as products of $S,T$. For example a matrix $S_7$ in $\mathrm{GL}(2,\Z/24\cdot 7 \Z)$ 
that reduces to $S$ modulo $7$ but to the identity modulo $24$ can be 
easily computed, $S_7=\begin{pmatrix} 49 & 48 \\ 120 & 49 \end{pmatrix}$.
This matrix has determinant $-3359 \equiv 1 \mod 24\cdot 7$. 
In order to decompose such a matrix as a product of $S,T$ elements 
we observe that left  multiplication by $S$ interchanges the rows of 
a $2\times2$ matrix and also multiplies the first row by $-1$ while 
left multiplication by $T^k$ adds the second row multiplied by $k$ to the 
first. So by successive divisions and interchanges we can arrive at an upper triangular 
matrix of the form $\begin{pmatrix} \pm 1 &   a \\ 0 & \pm 1 \end{pmatrix}$.
Then we can multiply by $S^2 =-\mathrm{Id}$ if necessary in order to arrive at
a matrix of the form $T^a$. This algorithm was explained to me by V. Metaftsis. 
For the cases $N=24\cdot 5$ and $N=24\cdot 7$ using 
magma \cite{magma} we were able to compute that 
{\tiny
\begin{eqnarray*}
T_3 &=& T^{-80},\\
T_8 &=&  T^{-15},\\
T_5 &=& T^{-24}, \\
S_3 &=& S \cdot T^{-10} \cdot S \cdot T^{18} \cdot S^{-1} \cdot T^{10} \cdot S^{-1} \cdot T^{-18} \cdot S \cdot 
T^{-10} \cdot S \cdot T^{-10} \cdot S \cdot T^{-21}  \cdot \\  & & \cdot S^{-1} \cdot T^9 \cdot S^{-1} \cdot T^{77} \cdot S \cdot T^5 \cdot S
\cdot T^2 \cdot 
S \cdot T^5 \cdot S,\\
S_8 &=& S^{-1} \cdot T^{-10} \cdot S \cdot T^{-10} \cdot S \cdot T^{-21} \cdot S^{-1} \cdot T^9 \cdot S^{-1} \cdot T^{59} \cdot 
S \cdot T^3 \cdot S \cdot T^{-4} \cdot S^{-1} \cdot  \\ && \cdot T^9 \cdot S^{-1} \cdot T^{-6} \cdot S \cdot T^8 \cdot S \cdot T^2 \cdot S \\
S_5 &=& S^{-1} 
\cdot T^{11} \cdot S \cdot T^{11} \cdot S^{-1} \cdot T^{11} \cdot S \cdot T^{-10} \cdot S \cdot T^{18} \cdot S^{-1} 
\cdot T^{10} \cdot S^{-1} \cdot 
T^{-18}  \cdot \\  & & \cdot 
 S \cdot T^{-10} \cdot S \cdot T^{-10} \cdot S \cdot T^{-21} \cdot S^{-1} \cdot T^9 \cdot S^{-1} \cdot T^{64}
\cdot S \cdot T^5 
\cdot S \cdot T^5 \cdot S  \\
\end{eqnarray*}
}
and
{\tiny 
\begin{eqnarray*}
T_3 &= & T^{-56},\\
T_8 &= & T^{-63},\\
T_7 & =& T^{-48}, \\
S_3 &= & S^{-1} \cdot T^{41} \cdot S \cdot T^{41} \cdot S^{-1} \cdot 
T^{101} \cdot S \cdot T^4 \cdot S \cdot T^4 \cdot S \cdot T^4 \cdot S\\
S_8 & =& S^{-1} \cdot T^{41} \cdot S \cdot T^{41} \cdot S^{-1} \cdot T^{41}
\cdot S \cdot T^{11} \cdot S^2 \cdot T^{-8} \cdot S \cdot T^{-40} \cdot  \\  && \cdot S^{-1}
\cdot T^{40} \cdot S^{-1} \cdot T^{19} \cdot S^{-1} \cdot T^{-37} 
\cdot S \cdot T^3 \cdot S \cdot T^3 \cdot S \cdot T^3 \cdot S \\
S_7 & =& 
S^{-1} \cdot T^{41} \cdot S \cdot T^{41} \cdot S^{-1} \cdot T^{41} 
\cdot S \cdot T^{-8} \cdot S \cdot T^{-40} \cdot S^{-1} \cdot T^{40} \cdot 
S^{-1} \cdot T^8 \cdot S \cdot \\  && \cdot T^{-8} \cdot  S \cdot T^{-40} \cdot S^{-1}
\cdot T^{40} \cdot  S^{-1} \cdot T^{22} \cdot S \cdot T^4 \cdot S \cdot 
T^2 \cdot S \cdot T^4 \cdot S, 
\end{eqnarray*}
}
respectivelly.


The action of the matrix $\begin{pmatrix}
             1 & 0 \\ 0 &d
            \end{pmatrix}$
is given by the action of the elements $\sigma_d \in \mathrm{Gal}(\Q(\zeta_N)/\Q)$
on the Fourier coefficients of the expansion at the cusp at infinity \cite{GeeBordeaux}. 
Since every element in $\mathrm{SL}(2,\Z/N\Z)$ can be written 
as a word in $S,T$ we obtain a function $\rho$ 
\begin{equation} \label{rhodef}
\xymatrix{
\left(\frac{\mathcal{O}}{N \mathcal{O}}\right)^*  \ar^\rho@/^1.5pc/[rr] \ar^\phi[r] & \mathrm{GL}(2,\Z/N\Z)  \ar[r] 
& \mathrm{GL}(V),
}
\end{equation}
where $\phi$ is the natural homomorphism given by theorem \ref{units2mat}.
\section{Finding class invariants} \label{3}
The map $\rho$ defined in eq. (\ref{rhodef}) in previous section is not a homomorphism. 
Indeed, if $e_1,\ldots,e_m$ is a basis of $V$, then the action of $\sigma$ is given in 
matrix notation as 
\[
 e_i \circ \sigma=\sum_{\nu=1}^m \rho(\sigma)_{\nu,i} e_\nu,
\]
and then since $(e_i \circ \sigma) \circ \tau= e_i \circ (\sigma\tau)$ we
obtain
\[
 e_i \circ (\sigma \tau) =\sum_{\nu,\mu=1}^m \rho(\sigma)_{\nu,i}^\tau
 \rho(\tau)_{\mu,\nu } e_\mu.
\]
Notice that the elements $\rho(\sigma)_{\nu,i} \in \Q(\zeta_N)$ and 
$\tau \in \mathrm{GL}(2,\Z/N\Z)$ acts on them as  well by the element
$\sigma_{\det(\tau)} \in \mathrm{Gal}(\Q(\zeta_N)/\Q)$.
So we arrive  at the following: 
\begin{proposition}
 The map $\rho$ defined in eq. (\ref{rhodef}) satisfies the  cocycle
condition
\begin{equation} \label{cocyc}
 \rho(\sigma\tau)=\rho(\tau)\rho(\sigma)^\tau 
\end{equation}
 and gives rise to a class in  
$H^1(G,\mathrm{GL}(V))$, where $G=(\mathcal{O}/N\mathcal{O})^*$.
The restriction of the map $\rho$ in the subgroup $H$ of $G$ defined by 
\[
 H:=\{x \in G :  \det(\phi(x))=1  \}
\]
is a homomorphism.
\end{proposition}
The basis elements $e_1,\ldots e_m$ are modular functions so there 
is a natural notion of multiplication for them.  We will   consider the 
 polynomial algebra  $\Q(\zeta_N)[e_1,\ldots, e_m]$. The group 
$H$ acts  on this algebra  in terms of the linear representation $\rho$
(recall that $\rho$ when restricted to $H$ is a homomorphism).

 Classical invariant 
theory provides us with effective methods 
(Reynolds operator method,linear algebra method \cite{KS97})  in order to compute the 
ring of invariants
$\Q(\zeta_N)[e_1,\ldots, e_m]^H$.
Also there  is a well defined action of the quotient group 
$G/H \cong \mathrm{Gal}(\Q(\zeta_N)/\Q)$ on $ \Q(\zeta_N)[e_1,\ldots, e_m]^H$.

Select the vector space $V_n$ of invariant polynomials of given degree $n$. 
\begin{remark}
For the applications in elliptic curves construction or in effective 
generation of the Hilbert class field
we have to take the smallest degree $n$ such that 
$V_n \neq \{0\}$. 
Indeed, it is known that there is a polynomial relation $F(f,j)$ among the functions $j,f$, 
where $j$ is the $j$-invariant, since the function field $F_N$
has transcendence degree $1$. It is known that this polynomial relation 
controls asymptotically the logarithmic  height $H(P_f)$, $H(P_j)$  of the minimal polynomial of $f$ and $j$
in the following way:
\[
 \lim_{h(j(\tau)\rightarrow \infty}
\frac{H(P_j)}{H(P_f)}=\frac{\deg_f F(f,j)}{\deg_j F(f,j)}=:r(f)
\]
where the limit is taken over all CM-points $\mathrm{SL}(2,\Z)\tau \in \mathbb{H}$ \cite{HinSil}.
So the best result comes when the $\deg_f F(f,j)$ is as big as possible. 
\end{remark}

The action of $G/H$ on $V_n$ gives rise to a cocycle
\[
 \rho' \in H^1(\mathrm{Gal}(\Q(\zeta_N))/\Q),V_n).
\]
The multidimensional Hilbert 90 theorem  asserts  that there is an element 
$P \in \mathrm{GL}(V_n)$ such that 
\begin{equation} \label{cobo}
 \rho' (\sigma) =  P^{-1} P^\sigma.
\end{equation}
Let $v_1,\ldots,v_\ell$ be a basis of $V_n$. 
The elements $v_i$ are by construction $H$ invariant. 
The elements $w_i:=v_i P^{-1}$ are $G/H$ invariant since
\[
 (v_i P^{-1}) \circ \sigma=( v_i \circ \sigma) (P^{-1})^{\sigma}=
v_i \rho(\sigma) (P^{-1})^\sigma=v_i P^{-1} P^{\sigma}  (P^{-1})^\sigma=
v_i P^{-1}.
\]
The above computation together with theorem \ref{units2mat} allows us to prove 
\begin{proposition}
Consider the polynomial ring  $\Q(\zeta_N)[e_1,\ldots,e_m]$ and the 
vector space  $V_n$  of  $H$-invariant homogenous polynomials of degree $n$. 
If  $P$ is a matrix such that eq. 
(\ref{cobo}) holds, then the images of a basis of $V_n$ under the action of $P^{-1 }$
are class invariants.         
\end{proposition}

How can we compute the matrix $P$ so that eq. (\ref{cobo}) holds?
We will use a version of  Glasby-Howlett probabilistic algorithm \cite{Glasby-Howlett}.
 We consider  the sum
\begin{equation} \label{BQ}
 B_Q:=\sum_{\sigma \in G/H} \rho(\sigma) Q^\sigma.
\end{equation}
If we manage to find a $2\times 2$ matrix in $\mathrm{GL}(2,\mathbb{Q}(\zeta_{N}))$
such that $B_Q$ is invertible then 
$P:=B_Q^{-1}$. 
Indeed, we compute that 
\begin{equation} \label{11}
 B_Q^{\tau} =\sum_{\sigma \in G/H} \rho(\sigma)^\tau  Q^{\sigma \tau}, 
\end{equation}
and  the cocycle condition $\rho(\sigma \tau)=\rho(\sigma)^\tau \rho(\tau)$, together with 
eq. (\ref{11}) allows us to write:
\[
 B_Q^{\tau}= \sum_{\sigma \in G/H} \rho(\sigma\tau) \rho(\tau)^{-1}  Q^{\sigma \tau}=
B_{Q} \rho_{\tau}^{-1}
\]
i.e.
\[
 \rho(\tau)= B_Q  \left(B_Q^{\tau} \right)^{-1}. 
\]
In order to obtain an invertible element $B_Q$ we feed eq. (\ref{11})
with random matrices $Q$ until $B_Q$ is invertible. Since non invertible 
matrices are rare (they form a Zariski closed subset in the space of matrices)
our first random choice of $Q$ always worked!

\section{Examples} \label{4}
Consider the generalised Weber functions 
$\frak{g}_0,\frak{g}_1,\frak{g}_2,\frak{g}_3$
 defined in the 
 work of A. Gee in \cite[p. 73]{GeePhD}
as 
\[
 \frak{g_0}(\tau)=\frac{\eta(\frac{\tau}{3})}{\eta(\tau)},\;
 \frak{g_1}(\tau)=\zeta_{24}^{-1}\frac{\eta(\frac{\tau+1}{3})}{\eta(\tau)},\;
\frak{g_2}(\tau)=\frac{\eta(\frac{\tau+2}{3})}{\eta(\tau)},\;
\frak{g_3}(\tau)=\sqrt{3}\frac{\eta(3\tau)}{\eta(\tau)},
\]
where $\eta$ denotes the Dedekind eta function:
\[
 \eta(\tau)=e^{2\pi i \tau/24} \prod_{n\geq 1}(1-q^n)\;\; \tau \in \mathbb{H}, q=e^{2\pi i \tau}.
\]
These are modular functions of level $72$.
In our previous work \cite{KoKo3} we investigated the action of the group 
$W_{N,\theta}$ for the $n\equiv 19 \mod 24$ case on the 
these modular functions 
and we showed  that the group $G:=W_{72,\theta}$ induces an action of the 
generalized symmetric group $\mu(12) \rtimes S_4$ on them. 
Any element $g$  of $G$ induces a matrix action by expressing
$ \frak{g_i}^g$, $i=0,1,2,3$  as a linear combination of the functions 
$\frak{g}_0,\frak{g}_1,\frak{g}_2,\frak{g}_3$. 
This way we obtain a map 
\begin{equation} \label{rho-def}
 \rho: G \rightarrow \mathrm{GL}(4,\mathbb{Q}(\zeta_{72}))=
\mathrm{Aut} \left(
\langle \frak{g}_0,\frak{g}_1,\frak{g}_2,\frak{g}_3 \rangle_{\mathbb{Q}(\zeta_{72})}
\right).
\end{equation}
In order to overcome the cocycle problem we raised everything to the 12-th power. 
This way the corresponding action 
\[
 \rho_{12}:G \rightarrow \mathrm{GL}(4,\mathbb{Q}(\zeta_{72}))=
\mathrm{Aut} \left(
\langle \frak{g}_0^{12},\frak{g}_1^{12},\frak{g}_2^{12},\frak{g}_3^{12}
 \rangle_{\mathbb{Q}(\zeta_{72})}
\right).
\]
becomes a group representation and we were able to find invariants of 
the action, that lead to class invariants by just applying the methods of
classical invariant theory for linear actions. 
This approach  has  a disadvantage; the class invariants we produce give 
rise to class polynomials with large coefficients. 

We consider the subgroup $H$ of $G$ defined by 
$H:=G \cap \mathrm{SL}\left( 2, \frac{\Z}{72 \Z} \right)$. When we restrict the 
map $\rho$ of eq. (\ref{rho-def}) we obtain a linear action and then 
we can construct the invariant polynomials of this action. 
Notice that there are no invariant polynomials of degree $1$ for $H$. 
But we can find invariant polynomials of degree $2$.  
For example for $n=-571$ the group $H$ has order $144$ and 
$G$ has order $3456$.  We find that the polynomials 
\[
 I_1:=\frak{g}_0 \frak{g}_2 + \zeta_{72}^6\frak{g}_1\frak{g}_3, \qquad
 I_2:=\frak{g}_0\frak{g}_3 + (-\zeta_{72}^{18} + \zeta_{72}^6)\frak{g}_1\frak{g}_2
\]
are indeed invariants of the 
action of $H$. Then we consider the action of $G/H$, which is an abelian group of order $24$
isomorphic to the group $\mathrm{Gal}(\mathbb{Q}(\zeta_{72})/\mathbb{Q})$. 
The quotient map gives rise to an action of 
\[
 \bar{\rho}:G/H \rightarrow \mathrm{GL}(2,\mathbb{Q}(\zeta_{72})=
\mathrm{Aut}( \langle I_1,I_2 \rangle_{\mathbb{Q}(\zeta_{72})}.
\]
The map $\bar{\rho}$ is again a cocycle in 
\[H^1(G/H,\mathrm{GL}(2,\mathbb{Q}(\zeta_{72}))=
H^1(\mathrm{Gal}(\mathbb{Q}(\zeta_{72})/\mathbb{Q}),
\mathrm{GL}(2,\mathbb{Q}(\zeta_{72}))
=0\]
by the multidimensional Hilbert 90 theorem. 
Therefore there is an element $P \in \mathrm{GL}(2,\mathbb{Q}(\zeta_{72}))$
such that 
\[
 \bar{\rho}(\sigma)=P^{\sigma}P^{-1}.
\]
The 
 elements 
$(I_1,I_2)\cdot P=:(e_1,e_2)$  given by
\[
   e_1:=(-12\zeta_{72}^{18} + 12\zeta_{72}^6)\frak{g}_0\frak{g}_3 + 
12\zeta_{72}^6\frak{g}_0 \frak{g}_3 + 12 \frak{g_1} \frak{g}_2 + 12 \frak{g}_1\frak{g}_3,
\]
  \[
e_2:=  12 \zeta_{72}^6 \frak{g}_1\frak{g}_2 + (-12 \zeta_{72}^{18} + 
12 \zeta_{72}^6) \frak{g}_0 \frak{g}_3 + (-12\zeta_{72} ^{12} + 12) \frak{g}_1 \frak{g}_3
 + 
        12\zeta_{72}^{12}\frak{g}_1\frak{g_3} 
\]
generate a $\Q$-vector space of class invariants.

In  table \ref{tab3} we write down the Hilbert polynomial corresponding to the $j$-invariant
the invariant corresponding to $\frak{g}_0^{12}\frak{g}_1^{12}+
\frak{g}_2^{12}\frak{g}_3^{12}$ and the polynomials corresponding to $e_1$ and 
$e_2$. We also examine the polynomial $\frak{g}_0^{12}\frak{g}_1^{12}+
\frak{g}_2^{12}\frak{g}_3^{12}$ since it is one of the few class invariants known in the 
$D \equiv 5\mod 24$ case.

\begin{table} 
\caption{Minimal polynomials using the $\frak{g}_0,\ldots,\frak{g}_3$ functions.} \label{tab3}
{\tiny
\begin{tabular}{|l|l|}
\hline
Invariant & polynomial \\
\hline
Hilbert &
$t^5 + 400497845154831586723701480652800t^4 +$ \\
&
    $818520809154613065770038265334290448384t^3 + $\\ 
& $4398250752422094811238689419574422303726895104t^2$\\ 
 &  $- 16319730975176203906274913715913862844512542392320t$ \\ 
 & $+ 15283054453672803818066421650036653646232315192410112$ \\
\hline
 &
 $ t^5 - 5433338830617345268674t^4$ + 
   $ 90705913519542658324778088 t^3$ \\ 
$\frak{g}_0^{12}\frak{g}_1^{12}+
\frak{g}_2^{12}\frak{g}_3^{12}$ & $-3049357177530030535811751619728 t^2$ \\
&$ - 
390071826912221442431043741686448t$ \\ & - 
    $12509992052647780072147837007511456$ \\
\hline 
$e_1$ & $t^5 - 936t^4 - 
    60912t^3 - 
2426112t^2 - 
40310784t - 
    3386105856$\\
\hline
$e_2$ & 
$t^5 - 1512t^4 - 
    29808t^3 + 
979776t^2 + 3359232t - 
    423263232$\\
\hline
\end{tabular}
}
\end{table}

\subsection{Generalized Weber functions}
The Weber and $\frak{g}_i$ functions are special cases of the so called {\em generalized
Weber functions} defined as:
\[
 \nu_{N,0}:=\sqrt{N} \frac{\eta \circ \begin{pmatrix} N & 0 \\ 0 & 1 \end{pmatrix}}{\eta}
 \mbox{ and } 
 \nu_{k,N}:=\frac{\eta \circ \begin{pmatrix} 1 & k \\ 0 & N \end{pmatrix}}{\eta}, 0 \leq k \leq N-1.
\]
These are known to be modular functions of level $24N$ \cite[th5. p.76]{GeePhD}.
Notice that $\sqrt{N} \in \mathbb{Q}(\zeta_N)\subset \mathbb{Q}(\zeta_{24\cdot N})$
 and an explicit 
expression of $\sqrt{N}$ in terms of $\zeta_N$ can be given by using Gauss sums
\cite[3.14 p. 228]{Fro-Tay}.

The group $\mathrm{SL}(2,\Z)$ acts on the $(N+1)$-th dimensional vector space generated by them. 
In order to describe this action we have to describe the action of the two 
generators $S,T$ of $\mathrm{SL}(2,\Z)$ given by $S:z\mapsto -\frac{1}{z}$ and 
$T:z \mapsto z+1$. 
Keep in mind that 
\[
 \eta \circ T (z)=\zeta_{24} \eta(z)  \mbox{ and }
 \eta \circ S(z)= \zeta_8^{-1} \sqrt{i z} \eta(z).
\]
We compute that (see also \cite[p.77]{GeePhD})
\[
 \nu_{N,0}\circ S= \nu_{0,N} \mbox{ and }  \nu_{N,0} \circ T=\zeta_{24}^{N-1} \nu_{N,0}, 
\]
\[
 \nu_{0,N}\circ S=\nu_{N,0} \mbox{ and }  \nu_{0,N} \circ T=\zeta_{24}^{-1} \nu_{1,N},
\]
for  $1 \leq k < N-1$ and $N$ is prime
\[
 \nu_{k,N} \circ S=\left(\frac{-c}{n} \right) i^{\frac{1-n}{2}} \zeta_{24}^{N(k-c)}
  \mbox{ and } \nu_{k,N} \circ T =\zeta_{24}^{-1} \nu_{k+1,N},
\]
where $c=-k^{-1} \mod N$. 
The computation of the action of $S$ on $\eta$ is the most difficult,  see
 \cite[eq. 8 p.443]{Hart}.

Assume that $N=5$ and $D=-91$.
We compute that the group $H$ of determinant $1$ has invariants
\[
\nu_{5,0}+(\zeta^{25} - \zeta^5)\nu_{3,5} \mbox{ and }
\nu_{0,5} + (\zeta^{31} - \zeta^{23} - \zeta^{19} - \zeta^{15} + \zeta^7 + \zeta^3)
\nu_{1,5}.
\]
Using the our method we arrive at the final invariants:
{\tiny
\begin{eqnarray*}
I_1 &= &(-1224\zeta^{28} + 612\zeta^{20} + 2740\zeta^{16} + 1516\zeta^4 - 612)\nu_{5,0}\\
 & & + 
(4256\zeta^{28} - 
        2128\zeta^{20} - 1516\zeta^{16} + 2740\zeta^4 + 2128)\nu_{0,5} \\
 & & + 
(-1224\zeta^{31} - 2740\zeta^{27} 
          + 612\zeta^{15} + 1224\zeta^{11} + 1516\zeta^3)\nu_{1,5} \\  &&  + 
(1516\zeta^{29} - 612\zeta^{25} + 
        1224\zeta^{13} - 1516\zeta^9 - 2740\zeta)\nu_{3,5},
\end{eqnarray*}
\begin{eqnarray*}
  I_2 &= & (-1952\zeta^{28} + 976\zeta^{20} + 2128\zeta^{16} + 176\zeta^4 - 976)\nu_{5,0}\\
& &
 + (2304\zeta^{28} - 
        1152\zeta^{20} - 176\zeta^{16} + 2128\zeta^4 + 1152)\nu_{0,5} \\
 & & + 
(-1952\zeta^{31} - 2128\zeta^{27} +
        976\zeta^{15} + 1952\zeta^{11} + 176\zeta^3)\nu_{1,5} \\ & & + 
(176\zeta^{29} - 976\zeta^{25} + 
        1952\zeta^{13} - 176\zeta^9 - 2128\zeta)\nu{3,5}.
\end{eqnarray*}
}
The $\Q$-vector space generated by these two functions consists of class functions. 
We can now compute the corresponding  polynomials:
\[t^2 - 3060t - 
        28090800 \mbox{ and }
    t^2 - 4880t - 71443200.
\]
Just for comparison the Hilbert polynomial corresponding to the $j$ invariant is:
 \[t^2 + 10359073013760t - 3845689020776448. 
\]
For $N=7$ and $n=91$ we have computed the invariants for the 
group $H$ of elements of determinant $1$ and we present the results in table \ref{table2}.  
\begin{table} 
 \label{table2}
\caption{Invariants for the group of elements of determinant 1 for $N=7$ and $n=91$}
{\tiny 
\begin{eqnarray*}I_1 &= & {{ \nu_{N,0}}}^{2}-{{ \nu_{0,N}}}^{2}+ \left( -{\zeta}^{42}+{\zeta}^{14}
 \right) {{ \nu_{1,N}}}^{2}+ \left( {\zeta}^{28}-1 \right) {{ \nu_{2,N}}}^{2}-
{\zeta}^{42}{{ \nu_{3,N}}}^{2}-{\zeta}^{14}{{ \nu_{5,N}}}^{2}+{{ \nu_{6,N}}}^{2},\\
I_2 &= & {
 \nu_{N,0}}\,{ \nu_{0,N}}+{\zeta}^{35}{ \nu_{N,0}}\,{ \nu_{1,N}}-{\zeta}^{28}{ \nu_{0,N}}
\,{ \nu_{2,N}}+{\zeta}^{35}{ \nu_{1,N}}\,{ \nu_{6,N}}-{\zeta}^{35}{ \nu_{2,N}}\,{ 
\nu_{5,N}}+ \\ & & \left( {\zeta}^{42}-{\zeta}^{14} \right) { \nu_{3,N}}\,{ \nu_{5,N}}+{
\zeta}^{21}{ \nu_{3,N}}\,{ \nu_{6,N}},\\
 I_3 & = &
{ \nu_{N,0}}\,{ \nu_{2,N}}+ \left( {\zeta}^{28}-
1 \right) { \nu_{N,0}}\,{ \nu_{6,N}}+{\zeta}^{7}{ \nu_{0,N}}\,{ \nu_{1,N}}+ \left( -{
\zeta}^{35}+{\zeta}^{7} \right) { \nu_{0,N}}\,{ \nu_{5,N}}\\ & &+{\zeta}^{28}{ \nu_{1,N}
}\,{ \nu_{3,N}}+ \left( {\zeta}^{35}-{\zeta}^{7} \right) { \nu_{2,N}}\,{ \nu_{3,N}
}-{\zeta}^{21}{ \nu_{5,N}}\,{ \nu_{6,N}},
\\
I_4 &= &
{ \nu_{N,0}}\,{ \nu_{3,N}}+ \left( -{\zeta}^{
42}+{\zeta}^{14} \right) { \nu_{N,0}}\,{ \nu_{5,N}}+{\zeta}^{42}{ \nu_{0,N}}\,{
 \nu_{3,N}}+ \left( -{\zeta}^{35}+{\zeta}^{7} \right) { \nu_{0,N}}\,{ \nu_{6,N}}-{
\zeta}^{42}{ \nu_{1,N}}\,{ \nu_{2,N}}+ \\  & &\left( {\zeta}^{35}-{\zeta}^{7}
 \right) { \nu_{1,N}}\,{ \nu_{5,N}}+ \left( -{\zeta}^{45}+{\zeta}^{37}+{\zeta}
^{33}-{\zeta}^{25}+{\zeta}^{17}+{\zeta}^{13}-{\zeta}^{5}-\zeta
 \right) { \nu_{2,N}}\,{ \nu_{6,N}},
\\
I_5 &= & 
{ \nu_{N,0}}\,{ \nu_{4,N}}+{\zeta}^{42}{ \nu_{0,N}}\,{
 \nu_{4,N}}+ \left( -{\zeta}^{45}+{\zeta}^{37}+{\zeta}^{33}-{\zeta}^{25}-{
\zeta}^{21}+{\zeta}^{17}+{\zeta}^{13}-{\zeta}^{5}-\zeta \right) { 
\nu_{1,N}}\,{ \nu_{4,N}}- \\ & & {\zeta}^{28}{ \nu_{2,N}}\,{ \nu_{4,N}}+ \left( {\zeta}^{35}-{
\zeta}^{7} \right) { \nu_{3,N}}\,{ \nu_{4,N}}+
 \\ & & \left( -{\zeta}^{45}+{\zeta}^{
37}+{\zeta}^{33}-{\zeta}^{25}+{\zeta}^{17}+{\zeta}^{13}-{\zeta}^{5}-
\zeta \right) { \nu_{4,N}}\,{ \nu_{5,N}}+{ \nu_{4,N}}\,{ \nu_{6,N}},\\
I_6& = &{{ \nu_{4,N}}}^{2}
\end{eqnarray*}
}
\end{table}
On  these invariants acts the group $\mathrm{Gal}(\Q(\zeta_{24\cdot 7})/\Q)$   
and we finally arrive at six invariant functions that over $\Q$ generate a family of 
invariant polynomials.
We present in \ref{table1}  just one of them.

\begin{table} 
\label{table1} 
\caption{An invariant comming from generalized Weber functions for $p=7$}
{\tiny
\begin{eqnarray*}
F_1&=&(-4\zeta^{44} + 4\zeta^{36} + 4\zeta^{32} + 4\zeta^{16} - 4\zeta^4 + 48)\nu_{N,0}^2  \\ &&+
 (4\zeta^{46} + 
        12\zeta^{42} - 4\zeta^{38} - 4\zeta^{34} - 4\zeta^{30} + 4\zeta^{26} + 4\zeta^{22} - 4\zeta^{14} + 
        4\zeta^6 + 4\zeta^2)\nu_{N,0}\nu_{0,N} \\ &&
 + (-8\zeta^{45} + 4\zeta^{41} + 8\zeta^{37} + 8\zeta^{33} - 
        8\zeta^{25} - 8\zeta^{21} + 12\zeta^{17} + 8\zeta^{13} - 12\zeta^5 - 8\zeta)\nu_{N,0}\nu_{1,N} \\ &&
 + 
        (-4\zeta^{36} + 16\zeta^{28} - 4\zeta^{16} + 4\zeta^8 + 4\zeta^4)\nu_{N,0}\nu_{2,N} \\ &&
 + (16\zeta^{47} - 
        28\zeta^{35} + 16\zeta^{27} - 16\zeta^{19} + 28\zeta^7 + 16\zeta^3)\nu_{N,0}\nu_{3,N} \\ && 
+ (-8\zeta^{38} - 
        8\zeta^{34} + 8\zeta^{26} + 16\zeta^{14} + 8\zeta^6)\nu_{N,0}\nu_{4,N} \\ && 
 + (12\zeta^{45} - 28\zeta^{37} - 
        12\zeta^{33} + 28\zeta^{25} - 12\zeta^{17} - 12\zeta^{13} + 12\zeta^5 + 28\zeta)\nu_{N,0}\nu_{5,N}\\ &&
 + 
        (-4\zeta^{44} + 4\zeta^{36} + 4\zeta^{32} + 4\zeta^{16} - 4\zeta^4 - 16)\nu_{N,0}\nu_{6,N}\\ &&
 + (4\zeta^{44} -
        4\zeta^{36} - 4\zeta^{32} - 4\zeta^{16} + 4\zeta^4 - 48)\nu_{0,N}^2 \\ &&+
 (-4\zeta^{43} + 16\zeta^{35} - 
        4\zeta^{23} + 4\zeta^{15} + 4\zeta^{11})\nu_{0,N}\nu_{1,N} \\ && 
+ (-4\zeta^{46} - 12\zeta^{42} + 4\zeta^{30} - 
        4\zeta^{22} + 12\zeta^{14} - 4\zeta^2)\nu_{0,N}\nu_{2,N} \\ &&
 + (16\zeta^{45} + 16\zeta^{41} - 16\zeta^{33} + 
        28\zeta^{21} - 16\zeta^{13})\nu_{0,N}\nu_{3,N} \\ &&
 + (-8\zeta^{44} + 8\zeta^{32} + 24\zeta^{28} + 8\zeta^8 - 
        24)\nu_{0,N}\nu_{4,N} \\ &&
+ (-4\zeta^{47} - 4\zeta^{43} + 4\zeta^{35} - 4\zeta^{27} - 4\zeta^{23} + 4\zeta^{19} +
        4\zeta^{15} + 4\zeta^{11} + 12\zeta^7 - 4\zeta^3)\nu_{0,N}\nu_{5,N} \\ &&
 + (16\zeta^{46} - 12\zeta^{42} - 
        16\zeta^{38} - 16\zeta^{34} - 16\zeta^{30} + 16\zeta^{26} + 16\zeta^{22} - 16\zeta^{14} + 
        16\zeta^6 + 16\zeta^2)\nu_{0,N}\nu_{6,N} \\ &&
+ (4\zeta^{46} - 48\zeta^{42} - 4\zeta^{30} + 4\zeta^{22} + 
        48\zeta^{14} + 4\zeta^2)\nu_{1,N}^2\\ &&
 + (-16\zeta^{45} - 16\zeta^{41} + 16\zeta^{33} - 28\zeta^{21} + 
        16\zeta^{13})\nu_{1,N}\nu_{2,N}\\ && 
+ (-4\zeta^{44} + 4\zeta^{32} + 16\zeta^{28} + 4\zeta^8 - 16)\nu_{1,N}\nu_{3,N} \\ && + 
        (8\zeta^{47} + 8\zeta^{43} - 8\zeta^{35} + 8\zeta^{27} + 8\zeta^{23} - 8\zeta^{19} - 8\zeta^{15} - 
        8\zeta^{11} - 16\zeta^7 + 8\zeta^3)\nu_{1,N}\nu_{4,N}\\ &&
 + (-16\zeta^{46} + 12\zeta^{42} + 16\zeta^{38} + 
        16\zeta^{34} + 16\zeta^{30} - 16\zeta^{26} - 16\zeta^{22} + 16\zeta^{14} - 16\zeta^6 - 
        16\zeta^2)\nu_{1,N}\nu_{5,N} \\ &&
+ (-8\zeta^{45} + 4\zeta^{41} + 8\zeta^{37} + 8\zeta^{33} - 8\zeta^{25} - 
        8\zeta^{21} + 12\zeta^{17} + 8\zeta^{13} - 12\zeta^5 - 8\zeta)\nu_{1,N}\nu_{6,N} \\ && + 
(4\zeta^{44} - 
        4\zeta^{32} + 48\zeta^{28} - 4\zeta^8 - 48)\nu_{2,N}^2 \\ &&
+ (4\zeta^{47} + 4\zeta^{43} - 4\zeta^{35} + 
        4\zeta^{27} + 4\zeta^{23} - 4\zeta^{19} - 4\zeta^{15} - 4\zeta^{11} - 12\zeta^7 + 
        4\zeta^3)\nu_{2,N}\nu_{3,N}\\ &&
 + (-8\zeta^{46} - 24\zeta^{42} + 8\zeta^{38} + 8\zeta^{34} + 8\zeta^{30} - 
        8\zeta^{26} - 8\zeta^{22} + 8\zeta^{14} - 8\zeta^6 - 8\zeta^2)\nu_{2,N}\nu_{4,N} \\ &&
 + (8\zeta^{45} - 
        4\zeta^{41} - 8\zeta^{37} - 8\zeta^{33} + 8\zeta^{25} + 8\zeta^{21} - 12\zeta^{17} - 8\zeta^{13}
 + 
        12\zeta^5 + 8\zeta)\nu_{2,N}\nu_{5,N}\\ &&
 + (16\zeta^{36} + 12\zeta^{28} + 16\zeta^{16} - 16\zeta^8 - 
        16\zeta^4)\nu_{2,N}\nu_{6,N}\\ &&
 + (4\zeta^{46} - 48\zeta^{42} - 4\zeta^{38} - 4\zeta^{34} - 4\zeta^{30} + 
        4\zeta^{26} + 4\zeta^{22} - 4\zeta^{14} + 4\zeta^6 + 4\zeta^2)\nu_{3,N}^2 \\ &&
+ (-16\zeta^{45} + 
        8\zeta^{41} + 16\zeta^{37} + 16\zeta^{33} - 16\zeta^{25} - 16\zeta^{21} + 24\zeta^{17} + 
        16\zeta^{13} - 24\zeta^5 - 16\zeta)\nu_{3,N}\nu_{4,N} \\ &&
 + (4\zeta^{36} - 12\zeta^{28} + 4\zeta^{16} - 
        4\zeta^8 - 4\zeta^4)\nu_{3,N}\nu_{5,N} \\ &&
+ (4\zeta^{47} + 8\zeta^{35} + 4\zeta^{27} - 4\zeta^{19} - 
        8\zeta^7 + 4\zeta^3)\nu_{3,N}\nu_{6,N} \\ &&
+ (-12\zeta^{36} + 12\zeta^{28} - 12\zeta^{16} + 12\zeta^8 + 
        12\zeta^4)\nu_{4,N}^2  \\ &&
+ (8\zeta^{47} + 16\zeta^{35} + 8\zeta^{27} - 8\zeta^{19} - 16\zeta^7 + 
        8\zeta^3)\nu_{4,N}\nu_{5,N} \\ && 
+ (-8\zeta^{38} - 8\zeta^{34} + 8\zeta^{26} + 16\zeta^{14} + 8\zeta^6)\nu_{4,N}\nu_{6,N} \\ && 
        + (-4\zeta^{38} - 4\zeta^{34} + 4\zeta^{26} - 52\zeta^{14} + 4\zeta^6)\nu_{5,N}^2 \\ &&
 + (16\zeta^{45} - 
        12\zeta^{37} - 16\zeta^{33} + 12\zeta^{25} - 16\zeta^{17} - 16\zeta^{13} + 16\zeta^5 + 
        12\zeta)\nu_{5,N}\nu_{6,N} \\ &&
+ (-4\zeta^{44} + 4\zeta^{36} + 4\zeta^{32} + 4\zeta^{16} - 4\zeta^4 + 
        48)\nu_{6,N}^2,
\end{eqnarray*}
} 
\end{table}

The corresponding polynomials for each class  invariant are
\[
\begin{array}{ll}
 t^2 + (420-8\sqrt{-91})t-20048, &t^2 + (672+40\sqrt{-91})t-57344,\\
 t^2 + (672+112\sqrt{-91})t-137984,&
t^2 + (1218+30\sqrt{-91})t-171136, \\
t^2 + (630-66\sqrt{-91})t-74592,&
t^2 + (798+54\sqrt{-91})t-91168
\end{array}
\]
Notice that the class polynomials have coefficients in $\mathcal{O}=\Z[\theta]$. Only if 
the value of the class function at $\theta$ is real, then  the class polynomial is in $\Z[t]$. For the 
construction of elliptic curves this is not a problem; we still can take the coefficients
modulo a prime ideal of $\mathcal{O}$ above $p$ and the values are either in 
$\mathbb{F}_p$ or in $\mathbb{F}_{p^2}$. 


%
%
\section{Comparison - conclusions}
How effective are the polynomials constructed by this method compared to other methods?
Let us compute the Hilbert class field of $\Q(\sqrt{-299})=\Q(\sqrt{D})$.
Using our method we arrive at the following invariants
\begin{eqnarray*}
I_1 &=& 12\zeta^{12}\frak{g}_0^2 + (-12\zeta^{12} + 12)\frak{g}_1^2 + 36\frak{g}_2\frak{g}_3,\\
I_2 &=&    36\zeta^{12}\frak{g}_0^2 + 12\frak{g}_2\frak{g}_3,\\
I_3 &=&    
24\zeta^{12}\frak{g}_0^2 + (-12\zeta^{18} + 24\zeta^6)\frak{g}_0\frak{g}_1 + 24\frak{g}_2\frak{g}_3,\\
I_4 &=&    12\zeta^{12}\frak{g}_0^2 + (-12\zeta^{18} + 24\zeta^6)\frak{g}_0\frak{g}_1 + 36\frak{g}_2\frak{g}_3
\end{eqnarray*}
with corresponding minimal polynomials
{\tiny
\begin{eqnarray*}
 P_1 &= & T^8 - 132T^7 - 3600T^6 - 1057536T^5 + 67578624T^4 + 2988223488T^3 + \\ &&
        159765073920T^2 + 5279816908800T + 59659100356608,\\ 
 P_2&= &   T^8 + (-36\sqrt{D} - 240)T^7 + (-1080\sqrt{D} + 37656)T^6 + (163296\sqrt{D} + 6612192)T^5 +\\ &&
        (19346688\sqrt{D} + 50471424)T^4 + (630415872\sqrt{D} - 19422706176)T^3 + \\ &&
        (-5925685248\sqrt{D} - 990861465600)T^2 + (-1731321298944\sqrt{D} + 1227669405696)T - \\ &&
        75541764243456\sqrt{D} - 516837998592,\\
P_3 &= &    T^8 + (-24\sqrt{D} - 612)T^7 + (-864\sqrt{D} + 27504)T^6 + (-82944\sqrt{D} + 4126464)T^5 + \\ &&
        (26002944\sqrt{D} + 3939840)T^4 + (376233984\sqrt{D} - 5667397632)T^3 + \\ &&
        (-14941863936\sqrt{D} - 771342372864)T^2 + (-264582070272\sqrt{D} + 27642125795328)T + \\ &&
        13454127267840\sqrt{D} - 355534235172864,\\
P_4 & =&    T^8 + (-12\sqrt{D} - 516)T^7 + (-504\sqrt{D} - 72)T^6 + (10368\sqrt{D} + 3680640)T^5 + \\ &&
        (20476800\sqrt{D} + 273849984)T^4 + (-430728192\sqrt{D} - 22758423552)T^3 + \\ &&
        (-39195518976\sqrt{D} - 559365875712)T^2 + (-114339299328\sqrt{D} + 36926863884288)T + \\ &&
        55540735672320\sqrt{D} + 99976764063744
\end{eqnarray*}
}
These are smaller that the coefficients of the Hilbert polynomial by a  factor of logarithmic height up to $6$
 but are not as efficient  as  the Ramanujan class invariant corresponding to
\[
 \frak{g}_2 \frak{g}_3=\frac{1}{48} I_2  -\frac{1}{16} I_3+ \frac{1}{16} I_4,
\]
which has a very small  minimal polynomial 
\[
 T^8 + T^7 - T^6 - 12T^5 + 16T^4 - 12T^3 + 15T^2 - 13T + 1.
\]
How can we select the most efficient class invariant?
Notice that every element in the $\Q$-vector space generated by the invariants $I_i$  constructed by our algorithm 
is a class invariant. Also all elements in the $\Z$-module generated by $I_i$ will give rise to class invariants with 
coefficients in $\mathcal{O}$. Of course (as the above example in $\Q(\sqrt{-299})$ indicates)
there might be elements of the form $\sum \lambda_i I_i$ with some $\lambda_i \in \Q-\Z$.
So far it seems a difficult problem how to select the most efficient class function among all 
class function. This problem is equivalent to minimizing the logarithmic height 
function on a lattice and seems out of reach for now. For the case of generalized Weber functions it  seems that 
monomials of the Weber invariants  are the best choices.
However there are cases, for example the $D \equiv 5 \mod 24$ case,  where no monomial invariants exist. 
Our method in this case provides a much better invariants  than the invariants constructed in 
\cite{KoKo3}, as one can see form table \ref{tab3}.

{\bf Acknowledgment:}
We would like to thank V. Metaftsis for his help on expressing matrices in 
$\mathrm{SL}(2,\Z/N\Z)$ in terms of $S,T$. Also, we would like  to thank 
the magma algebra team for providing us with a free extension of the 
license to their system.

\providecommand{\bysame}{\leavevmode\hbox to3em{\hrulefill}\thinspace}
\providecommand{\MR}{\relax\ifhmode\unskip\space\fi MR }
\providecommand{\MRhref}[2]{%
  \href{http://www.ams.org/mathscinet-getitem?mr=#1}{#2}
}
\providecommand{\href}[2]{#2}

\end{document}